\documentclass{amsart}
\usepackage{mathtools, microtype, mathrsfs, xfrac, hyperref, amssymb, enumerate, xspace, stmaryrd}
\usepackage[latin1]{inputenc}
\usepackage{tikz-cd}
\usepackage{todonotes}

\numberwithin{equation}{section}

\numberwithin{subsection}{section}

\allowdisplaybreaks[1]

\newenvironment{enumeratea}
{\begin{enumerate}[\upshape (a)]}
{\end{enumerate}}

\newenvironment{enumerate1}
{\begin{enumerate}[\upshape (1)]}
{\end{enumerate}}

\newtheorem*{namedtheorem}{\theoremname}
\newcommand{\theoremname}{testing}
\newenvironment{named}[1]{\renewcommand\theoremname{#1}
\begin{namedtheorem}}
{\end{namedtheorem}}

\newtheorem{theorem}{Theorem}[section]
\newtheorem{proposition}[theorem]{Proposition}
\newtheorem{proposition-definition}[theorem]
{Proposition-Definition}
\newtheorem{corollary}[theorem]{Corollary}
\newtheorem{lemma}[theorem]{Lemma}

\theoremstyle{definition}
\newtheorem{definition}[theorem]{Definition}
\newtheorem{notation}[theorem]{Notation}

\newtheorem{examples}[theorem]{Examples}

\newtheorem{remarks}[theorem]{Remarks}

\theoremstyle{remark}

\renewcommand{\mathcal}{\mathscr}

 \newcommand\cB{\mathcal{B}}
 
\newcommand\cE{\mathcal{E}} 
\newcommand\cG{\mathcal{G}}

\newcommand\cM{\mathcal{M}} \newcommand\cN{\mathcal{N}}
\newcommand\cO{\mathcal{O}}

\newcommand\cU{\mathcal{U}} 
 \newcommand\cX{\mathcal{X}}

\renewcommand\AA{\mathbb{A}} 
\newcommand\CC{\mathbb{C}}

 \newcommand\PP{\mathbb{P}}

 \newcommand\bX{\mathbf{X}}

\newcommand\rA{\mathrm{A}} 
\newcommand\rC{\mathrm{C}}

\newcommand\rM{\mathrm{M}}

\newcommand\rS{\mathrm{S}}

\newcommand\rmc{\mathrm{c}}

 \newcommand\bfp{\mathbf{p}}

\newcommand\arr{\ifinner\to\else\longrightarrow\fi}

\newcommand\arrto{\ifinner\mapsto\else\longmapsto\fi}

\newcommand{\xarr}{\xrightarrow}

\renewcommand\H{\operatorname{H}}

\newcommand\op{^{\mathrm{op}}}

\newcommand{\eqdef}{\mathrel{\smash{\overset{\mathrm{\scriptscriptstyle def}} =}}}

\def\displaytimes_#1{\mathrel{\mathop{\times}\limits_{#1}}}

\def\displayotimes_#1{\mathrel{\mathop{\bigotimes}\limits_{#1}}}

\renewcommand\hom{\operatorname{Hom}}

\newcommand\aut{\operatorname{Aut}}

\newcommand\out{\operatorname{Out}}

\newcommand\spec{\operatorname{Spec}}

\newcommand\pr{\operatorname{pr}}

\newcommand\indlim{\varinjlim}

\newcommand{\cat}[1]{(\mathrm{#1})}

\newcommand{\catset}{(\mathrm{Set})}

\newcommand\double{\mathbin{\rightrightarrows}}

\newcommand{\underisom}{\mathop{\underline{\mathrm{Isom}}}\nolimits}

\newcommand{\underaut}{\mathop{\underline{\mathrm{Aut}}}\nolimits}

\newlength{\ignora}

\newcommand{\mmu}{\boldsymbol{\mu}}

\newcommand{\GL}{\mathrm{GL}}

\newcommand{\PGL}{\mathrm{PGL}}

\DeclareFontFamily{U}{mathx}{\hyphenchar\font45}
\DeclareFontShape{U}{mathx}{m}{n}{
      <5> <6> <7> <8> <9> <10>
      <10.95> <12> <14.4> <17.28> <20.74> <24.88>
      mathx10
      }{}
\DeclareSymbolFont{mathx}{U}{mathx}{m}{n}
\DeclareFontSubstitution{U}{mathx}{m}{n}
\DeclareMathAccent{\widecheck}{0}{mathx}{"71}
\DeclareMathAccent{\wideparen}{0}{mathx}{"75}

\renewcommand{\epsilon}{\varepsilon}

\newcommand{\cha}{\operatorname{char}}

\newcommand{\aff}[1][k]{(\mathrm{Aff}/#1)}
\newcommand{\as}[1][k]{(\mathrm{FAS}/#1)}
\newcommand{\pas}[1][k]{(\mathrm{PFAS}/#1)}

\begin{document}

\title[Fields of moduli and the arithmetic of tame quotient singularities]{Fields of moduli and the arithmetic\\of tame quotient singularities}

\author[Bresciani]{Giulio Bresciani}
\author[Vistoli]{Angelo Vistoli}

\address{Scuola Normale Superiore\\Piazza dei Cavalieri 7\\
56126 Pisa\\ Italy}
\email[Vistoli]{angelo.vistoli@sns.it}
\email[Bresciani]{giulio.bresciani@gmail.com}

\thanks{The second author was partially supported by research funds from Scuola Normale Superiore, project \texttt{SNS19\_B\_VISTOLI}, and by PRIN project ``Derived and underived algebraic stacks and applications''.  The paper is based upon work partially supported by the Swedish Research Council under grant no. 2016-06596 while the second author was in residence at Institut Mittag-Leffler in Djursholm}


\begin{abstract}
	Given a perfect field $k$ with algebraic closure $\overline{k}$ and a variety $X$ over $\overline{k}$, the field of moduli of $X$ is the subfield of $\overline{k}$ of elements fixed by field automorphisms $\gamma\in\operatorname{Gal}(\overline{k}/k)$ such that the twist $X_{\gamma}$ is isomorphic to $X$. The field of moduli is contained in all subextensions $k\subset k'\subset\overline{k}$ such that $X$ descends to $k'$. In this paper we extend the formalism, and define the field of moduli when $k$ is not perfect.
	
	Furthermore, Dèbes and Emsalem identified a condition that ensures that a smooth curve is defined over its field of moduli, and prove that a smooth curve with a marked point is always defined over its field of moduli. Our main theorem is a generalization of these results that applies to higher dimensional varieties, and to varieties with additional structures.

	In order to apply this, we study the problem of when a rational point of a variety with quotient singularities lifts to a resolution. As a consequence, we prove that a variety $X$ of dimension $d$ with a smooth marked point $p$ such that $\aut(X,p)$ is finite, étale and of degree prime to $d!$ is defined over its field of moduli.
\end{abstract}

\maketitle

\section{Introduction}

The concept of \emph{field of moduli} was introduced by Matsusaka in \cite{matsusaka-field-of-moduli}, and considerably clarified by Shimura in \cite{shimura-automorphic}. Suppose that $k$ is a field with algebraic closure $\overline{k}$. Let us assume for simplicity that $k$ is perfect. An algebraic variety $X$, perhaps with additional structure, like a polarization, or a marked point, will be defined over some intermediate field $k\arr \ell \arr \overline{k}$ that is finite over $k$. If $\Gamma$ is the Galois group of $\overline{k}$ over $k$, call $\Delta \subseteq \Gamma$ the subgroup formed by elements $\gamma \in \Gamma$ such that the twist $X_{\gamma}$ is isomorphic to $X$ as a $k$-scheme, possibly with its additional structure. Then $\Delta$ is an open subgroup of $\Gamma$; the field of moduli of $X$ is the fixed subfield $\overline{k}^{\Delta}$. It is contained in every field of definition of $X$.

If $X$ has a finite group of automorphisms, and is one of a class of varieties with finite automorphism groups parametrized by a coarse moduli space $M \arr \spec k$ (for example, smooth curves of genus at least $2$), then the field of moduli has a natural interpretation as the residue field of the image of the morphism $\spec \overline{k} \arr M$ corresponding to $X$.

One basic question is: when is $X$ defined over its field of moduli? This problem has been the subject of a considerable amount of literature over the years.

An important early example is due to Goro Shimura \cite{shimura-field-rationality-abelian}. Let $A_{g}$ be the moduli space of abelian varieties of genus~$g$ over $\CC$, and call $K$ its field of rational functions. Let $X$ be the corresponding generic abelian variety defined over $\overline{K}$; its field of moduli is $K$. Then Shimura proved that $X$ is defined over $K$ if and only $g$ is odd.

In the case that $X$ is a smooth curve, an important advance is due to Pierre Dèbes and Michel Emsalem \cite{debes-emsalem}. 

Let $X$ be as above; assume that the group $\aut X$ of automorphisms of $X$ over $\overline{k}$ is finite. Consider the group $\Delta$ above; for each $\delta \in \Delta$ we have an isomorphism $X_{\delta} \simeq X$, well defined up to an automorphism of $X$ over $\overline{k}$. This descend to an canonical automorphism of $X/\aut X$, defining an action of $\Delta$ on $X$, compatible with it action on $\overline{k}$; by Galois descent this defines a scheme $X^{\rmc}$ over $k$, a form of $X/\aut X$, which we call the \emph{compression} of $X$. If $X$ is a smooth curve, so is $X^{\rmc}$.

\begin{named}{Theorem}[Dèbes--Emsalem]
Assume that $X$ is a smooth curve. If $X^{\rmc}(k) \neq \emptyset$, then $X$ is defined over its field of moduli.
\end{named}

(See \cite{giulio-moduli-divisor} for applications of this result.)

Let us describe the content of this paper.

\subsubsection*{The definition of field of moduli} The definition of field of moduli for an object defined over $\overline{k}$ is present in the literature in two particular cases, as explained above: when $k$ is perfect, and when $X$ is in a class of object having a coarse moduli space. This last hypothesis is not a very natural one: for example, one could be interested in the field of moduli of non-polarized abelian varieties, or K3 surfaces. In \S\ref{sec:field-moduli} we give a somewhat more general and flexible formalism for defining fields of moduli and residual gerbes under very weak hypotheses. In particular we adapt it to non-perfect fields and objects with non-reduced automorphism group schemes, using the fppf topology instead of the Galois group.

In \S\ref{sec:nonabelian} we draw some consequences using results in the literature; this part is certainly well known to the experts, and essentially contained in \cite{debes-emsalem}; but it seems useful to have a detailed explanation in our generality.

In the rest of the introduction we will assume $\cha k = 0$; this will simplify the statements considerably. We refer to the main body of the paper for precise statements in arbitrary characteristic.

\subsubsection*{The Dèbes--Emsalem theorem in arbitrary dimension} If $X$ is as in the introduction, if $X$ is singular, or higher-dimensional, it is not true that if $X^{\rmc}(k) \neq \emptyset$, then $X$ is defined over its field of moduli. Our main result (Theorem~\ref{thm:moduli}) is the following: let $\widetilde{X}$ be a resolution of singularities of $X^{\rmc}$. If $\widetilde{X}(k) \neq \emptyset$, then $X$ is defined over its field of moduli. 

When $X$ is a smooth curve, we have that the compression $X^{\rmc}$ is smooth, so this recovers the result of Dèbes--Emsalem. Dèbes and Emsalem mention the fact that their methods can be generalized to curves with a structure such as pointed curves \cite[Remark 3.2.(b)]{debes-emsalem}, but they do not give details on how to do it. So in dimension $1$, we are essentially clarifying what curves with a structure are, and checking that the theorem of Dèbes--Emsalem holds for them. More importantly, we are able to generalize their result to arbitrary dimension.

Dèbes and Emsalem also observed that a suitable form of their theorem implies that every pointed curve of genus $\ge 2$ is defined over its field of moduli \cite[Corollary~5.4]{debes-emsalem}. This fails, however, in higher dimension: smooth pointed varieties are not necessarily defined over their fields of moduli. An example is Shimura's result on generic abelian varieties explained in the introduction. In order to apply the result above we need to ensure that a rational point of the compression $X^{\rmc}$ lifts to a resolution of singularities.

More precisely, we need conditions that ensure that a resolution of singularities $\widetilde{X} \arr X^{\rmc}$ has a rational point. In the case of pointed varieties this translates into the following question. Suppose that $X$ is a variety over $\overline{k}$ and $G$ is a finite group acting on $X$ with a smooth fixed point $\overline{p} \in X(\overline{k})$. Let $X^{\rmc}$ be a form of $X/G$ defined over $k$, with a rational point $p \in X^{\rmc}(k)$ lifting to $\overline{p}$. If $\widetilde{X} \arr X^{\rmc}$ is a resolution of singularities, under what conditions does it follow that there is a rational point of $\widetilde{X}$ lying over $p$?

\subsubsection*{The arithmetic of quotient singularities} In \S\ref{sec:arithmetic} we introduce two related concepts. 

One is that of \emph{$R$-singularity} (Definition~\ref{def:R-singularity}). An $R$-singularity is a pair $(S,s)$, where $S$ is a variety over a field $K$ with quotient singularities and $s \in S(K)$, such that, in particular, if $k \subseteq K$ is a subfield, $(S', s')$ is a form of $(S,s)$ defined over $k$, and $\widetilde{S}' \arr S$ is a resolution of singularities, then $\widetilde{S}'$ has a $k$-rational point over $s'$.

The other key definition is that of $R_{d}$-group (Definition~\ref{def:Rd}): if $d$ is a positive integer, a finite group $G$ is $R_{d}$ if whenever it acts faithfully on a smooth variety $X$ with a fixed rational point $x \in X(k)$, the pair $(X/G, [x])$, where $[x]$ is the image of $x$, is an $R$-singularity. It follows from our main theorem that if $X$ is a smooth variety over $\overline{k}$ and $\overline{p} \in X(\overline{k})$ is a $\overline{k}$-point, and $\aut (X, \overline{p})$ is an $R_{d}$ group, then $(X, \overline{p})$ is defined over its field of moduli.

Not all finite groups are $R_{d}$: for example, a cyclic group of order~$2$ is not $R_{d}$ for any $d \geq 2$. Here we present two results, showing that this is not an empty definition.

The first (Theorem~\ref{thm:Rd-1}) shows that there are infinitely many groups that are $R_{d}$ for all $d$. The second (Theorem~\ref{thm:Rd-2}) says that any finite group of order prime to $d!$ is $R_{d}$. As a consequence (Theorem~\ref{thm:pointed!}), a $d$-dimensional variety $X$ with a smooth marked point $p\in X$ such that $\aut(X,p)$ is finite of degree prime to $d!$ is defined over its field of moduli.

There is much more that one could say about $R_{d}$-groups. The forthcoming paper \cite{giulio-singularities} will contain a complete classification of $R_{2}$-groups. The case $d > 2$ seems much harder; hopefully it will be the subject of further work.

Another result along these lines is Corollary~\ref{cor:isolated-C2}, stating that if $X$ is an odd-dimensional variety over $\overline{k}$ with a smooth marked point $\overline{p} \in X(\overline{k})$, and the automorphism group of $(X,\overline{p})$ is cyclic of order $2$ and has $\overline{p}$ as an isolated fixed point, then $(X,\overline{p})$ is defined over its field of moduli. This would be false without assuming that $\overline{p}$ is an isolated fixed point, as the cyclic group $\rC_{2}$ is not $R_{d}$ for any $d \geq 2$. This vastly generalizes Shimura's result that odd dimensional generic principally polarized abelian varieties are defined over their field of moduli.

\medskip

For the second and third part, the fundamental tool is the Lang--Nishimura theorem for tame stacks proved in \cite{giulio-angelo-valuative} (see Theorem~\ref{LN}).

\subsection*{Acknowledgments} We are grateful to Dan Abramovich and Pierre Dèbes for some very useful discussions, and to J\'anos Koll\'ar for pointing out Lemma~\ref{lem:abelian} to us. We thank all of them for their interest in our work.

\section{Notations and conventions}

We will follow the conventions of \cite{knutson} and \cite{laumon-moret-bailly}; so the diagonals of algebraic spaces and algebraic stacks will be separated and of finite type. In particular, every algebraic space will be \emph{decent}, in the sense of \cite[Definition~03I8]{stacks-project}.

We will follow the terminology of \cite{dan-olsson-vistoli1}: a \emph{tame stack} is an algebraic stack $X$ with finite inertia, such that its geometric points have linearly reductive automorphism group. This is equivalent to requiring that $X$ is étale locally over its moduli space a quotient by a finite, linearly reductive group scheme \cite[Theorem 3.2]{dan-olsson-vistoli1}.

If $k$ is a field and $G \arr \spec k$ is a group scheme, we denote by $\cB_{k}G$ the classifying stack of $G$, whose objects are $G$-torsors.

\section{Fields of moduli}\label{sec:field-moduli}

Let $k$ be a field, $\aff$ the category of affine $k$-schemes. All stacks will be fppf stacks over $\aff$. If $\cM$ is such a stack, and $R$ is a $k$-algebra, we set $\cM(R) \eqdef \cM(\spec R)$; if $R \arr S$ is a morphism of $k$-algebras and $\xi$ is an object of $\cM(R)$, we denote by $\xi_{S}$ the pullback of $\xi$ to $\cM(S)$ via the induced morphism $\spec S \arr \spec R$.

Recall that a stack $\cM \arr \aff$ is \emph{locally finitely presented} if whenever $\{A_{i}\}_{i \in I}$ is a filtered inductive system of $k$-algebras, the induced functor $\indlim_{i}\cM(A_{i}) \arr \cM(\indlim_{i}A_{i})$ is an equivalence. In particular we have the notion of a locally finitely presented functor $\aff \arr \cat{Set}$.

Suppose that $K$ is an extension of $k$. We define, as usual, the automorphisms functor
   \[
   \underaut_{K}\xi\colon \aff[K]\op \arr \cat{Grp}\,,
   \]
where $\cat{Grp}$ is the category of groups, as the functor sending an affine $K$-scheme $S$ into the group of automorphisms of the pullback $\xi_{S}$. 

If $\xi$ and $\eta$ are two objects of $\cM(K)$, we denote by
   \[
   \underisom_{K\otimes_{k}K}(\pr_{1}^{*}\xi, \pr_{2}^{*}\eta)\colon \aff[S]\op \arr \cat{Set}
   \]
the functor of isomorphisms of the pullbacks of $\xi$ and $\eta$ along the two projections $\pr_{1}$, $\pr_{2}\colon \spec K \times_{\spec k} \spec K = \spec(K\otimes_{k}K) \arr \spec K$.

If $\cM$ is locally finitely presented, the two functors above are locally finitely presented. Recall that, in our terminology, algebraic spaces are quasi-separated: under this assumption, group algebraic spaces are separated schemes \cite[\href{https://stacks.math.columbia.edu/tag/08BH}{Tag 08BH}, \href{https://stacks.math.columbia.edu/tag/0B8G}{Tag 0B8G}]{stacks-project}. In particular, the functor $\underaut_{K}\xi$ is an algebraic space of finite type if and only if it is a group scheme of finite type.

\begin{lemma}
Let $\cM \arr \aff$ be a locally finitely presented fppf stack, $K$ an extension of $k$, $\xi$ and $\eta$ two objects of $\cM(K)$. Let $K'$ be an algebraic extension of $K$; then we have the following two equivalences.

\begin{enumerate1}

\item The functor $\underisom_{K\otimes_{k}K}(\pr_{1}^{*}\xi, \pr_{2}^{*}\eta)$ is an algebraic space of finite type over $K$ if and only if $\underisom_{K'\otimes_{k}K'}(\pr_{1}^{*}\xi_{K'}, \pr_{2}^{*}\eta_{K'})$ is an algebraic space of finite type over $K'$.

\item The functor $\underaut_{K}\xi$ is a group scheme of finite type if and only if $\underaut_{K'}\xi'$ is a group scheme of finite type.

\end{enumerate1}
\end{lemma}

Both part are immediate consequences of the following, applied to the cases $R = K$ $R' = K'$, and $R = K\otimes_{k}K$, $R' = K'\otimes_{k}K'$.

\begin{lemma}
Let $R$ be a commutative ring, $\{R_{i}\}_{i \in I}$ a filtered inductive system of finitely presented $R$-algebras with faithfully flat transition functions. Set $R' \eqdef \indlim_{i}R_{i}$. Let $F\colon \aff[R]\op \arr \catset$ be an fppf sheaf that is locally of finite presentation, and call $F'$ the composite $\aff[R']\op \arr \aff[R]\op \arr \catset$, where the functor $\aff[R']\op \arr \aff[R]\op$ is given by restriction of scalars. Then $F$ is a finitely presented algebraic space over $R$ if and only if $F'$ is a finitely presented algebraic space over $R'$.
\end{lemma}

\begin{proof}
If $F$ is a finitely presented algebraic space then $F' = \spec R' \times_{\spec R}F$ is also a finitely presented algebraic space. So, assume that $F'$ is a finitely presented algebraic space.

If $R'$ is finitely presented over $R$, then the result follows from Artin's theorem \cite[Corollaire~10.4]{laumon-moret-bailly}. So, denote by $F_{i}$ the pullback of $F$ to $\spec R_{i}$; it is enough to show that $F_{i}$ is a finitely presented algebraic space for some $i$. Since $F'$ is finitely presented, it will be obtained by pullback from a finitely presented algebraic space $G$ over some $R_{i}$; we can replace $R,F$ with $R_{i},F_{i}$, and assume that $G$ is defined over $R$. Then it is enough to show that for some $i$, the pullback of $G$ to $\spec R_{i}$ is isomorphic to $F_{i}$.

We claim that the isomorphism $G_{R'} \simeq F_{R'} = F'$ comes from a morphism $G_{R_{i}} \arr F_{i}$ for some $i$. For this, choose an affine scheme $U$ with an étale surjective map $U \arr G$, and an affine scheme $V$ with an étale surjective map $V \arr U\times_{G}U$. Since $F$ is an fppf sheaf, for any $R$-algebra $S$ the set of morphism $\hom_{S}(G_{S},F_{S})$ is the equalizer of the two maps $F(U_{S}) \arr F(V_{S})$; since $F$ is locally of finite presentation, and filtered colimits commute with equalizers, we easily get $\hom_{R'}(G_{R'},F_{R'}) = \indlim_{i}\hom_{R_{i}}(G_{R_{i}}, F_{i})$, hence the thesis. So, by replacing $R$ with $R_{i}$ we may assume that there exists a morphism $G \arr F$ that pulls back to the isomorphism $G_{R'} \simeq F_{R'}$; we need to check that this is an isomorphism. This is straightforward, and left to the reader.
\end{proof}

\subsection{The residual gerbe}

\begin{notation}
In this subsection we will always use the following notation: $\cM$ will be an fppf stack over $\aff$, locally of finite presentation, $K$ an algebraic extension of $k$, $\xi$ an object of $\cM(K)$. We will also interpret $\xi$ as a morphism $\xi\colon \spec K \arr \cM$.

\end{notation}

Assume that $K$ is finite over $k$. We define the \emph{residual gerbe} $\cG_{\xi}$ of $\xi$ in $\cM$ as in \cite[\S11]{laumon-moret-bailly}, that is, we take $\cG_{\xi}$ to be the fppf image of $\xi$ in $\cM$. In other words, $\cG_{\xi}$ is the full fibered subcategory of $\cM$ such that an object $\eta$ in $\cM(S)$ is in $\cG_{\xi}(S)$ if there exists an fppf cover $S' \arr S$ and a morphism of $k$-schemes $S' \arr \spec K$ such that the pullback $\eta_{S'}$ and $\xi_{S'}$ are isomorphic in $\cM(S')$. It is immediate to check that $\cG_{\xi}$ is an fppf stack.

Notice that by definition $\xi\colon \spec K \arr \cM$ factors through $\cG_{\xi}$; in fact $\cG_{\xi} \subseteq \cM$ is the smallest fppf substack of $\cM$ through which $\xi$ factors.

\begin{lemma}\label{lem:same-gerbe-finite}
Let $L$ be a finite extension of $K$. Then $\cG_{\xi_{L}} = \cG_{\xi}$.
\end{lemma}

\begin{proof}
This is clear from the fact that $\spec L \arr \spec K$ is an fppf cover.
\end{proof}

In the general case $K$ is an \emph{algebraic} extension of $k$, and $\xi$ an object of $\cM(K)$ as before. Since $\cM$ is locally of finite presentation, there is a factorization $\spec K \arr \spec L \xarr{\xi'}$ of $\xi\colon \spec K \arr \cM$ with $L$ finite over $k$; we define $\cG_{\xi}$ to be $\cG_{\xi'}$. Because of Lemma~\ref{lem:same-gerbe-finite}, this is independent of the factorization. It is immediate to show that the analogue of Lemma~\ref{lem:same-gerbe-finite} holds when $K/k$ and $L/K$ are only assumed to be algebraic.

\begin{proposition}\label{prop:same-gerbe}
Let $L$ be an algebraic extension of $K$. Then $\cG_{\xi_{L}} = \cG_{\xi}$.
\end{proposition}

\begin{proof}
Let $E$ be an intermediate extension $k \subseteq E \subseteq K$, finite over $k$, such that $\xi\colon \spec K \arr \cM$ factors as $\spec K \arr \spec E \xarr{\xi'} \cM$. By definition we have $\cG_{\xi} = \cG_{\xi'} = \cG_{\xi_{L}}$.
\end{proof}

\begin{definition}\label{def:algebraic}
We say that $\xi$ is \emph{algebraic} if the residual gerbe $\cG_{\xi}$ is an algebraic stack.
\end{definition}

From Proposition~\ref{prop:same-gerbe} we obtain the following.

\begin{proposition}
Let $L$ be an algebraic extension of $K$. Then $\xi$ is algebraic if and only if $\xi_{L}$ is algebraic.
\end{proposition}

\begin{proposition}\label{prop:algebraic}
The following conditions are equivalent.

\begin{enumerate1}

\item The object $\xi$ is algebraic.

\item The functor
   \[
   \underisom_{K\otimes_{k}K}(\pr_{1}^{*}\xi, \pr_{2}^{*}\xi)\colon \aff[K\otimes_{k}K]\op \arr \cat{Set}
   \]
is an algebraic space of finite type.

\end{enumerate1}
\end{proposition}

\begin{proof}
(1)${}\implies{}$(2). The functor in question is the fibered product $\spec K\times_{\cM}\spec K$; since $\cG_{\xi}$ is a full subcategory of $\cM$ and $\spec K \arr \cM$ factors through $\cG_{\xi}$, we have $\spec K\times_{\cM}\spec K = \spec K\times_{\cG_{\xi}}\spec K$, and the result follows.

(2)${}\implies{}$(1). Set  $S \eqdef \spec K$ and $R \eqdef \underisom_{K\otimes_{k}K}(\pr_{1}^{*}\xi, \pr_{2}^{*}\xi)$. We obtain an fppf groupoid $R \double S$; since $S \arr \cG_{\xi}$ is an fppf cover, we have an equivalence of $\cG_{\xi}$ with the quotient stack $[R \double S]$. Then it follows from Artin's theorem \cite[Corollaire~10.6]{laumon-moret-bailly} that $\cG_{\xi}$ is an algebraic stack.
\end{proof}

If $\xi$ is algebraic, then the functor $\underaut_{K}\xi\colon \aff[K]\op \arr \cat{Grp}$ is a group scheme of finite type, as it is the restriction of $\underisom_{K\otimes_{k}K}(\pr_{1}^{*}\xi, \pr_{2}^{*}\xi)$ along the diagonal $\spec K \subseteq \spec K\otimes_{k}K$. We do not know whether the converse holds in general; but it does if $K/k$ is separable.

\begin{proposition}\label{prop:aut-algebraic}
Assume that the extension $K/k$ is separable, and that $\underaut_{K}\xi$ is a group scheme of finite type. Then $\xi$ is algebraic. 
\end{proposition}

\begin{proof}
We can assume that $K$ is finite over $k$. Set $G \eqdef \underaut_{K}\xi$ and Â $R \eqdef \underisom_{K\otimes_{k}K}(\pr_{1}^{*}\xi, \pr_{2}^{*}\xi)$; there is a natural right action of $G$ on $R$, by right composition. Assume that $G$ is representable, and let us show that $R$ is representable.

We have $K\otimes_{k}K = L_{1}\times \dots \times L_{r}$, where the each of the $L_{i}$ is a finite separable extension of $k$. Let us show that for each $i$ the restriction $R_{L_{i}}$ of $R$ to $\spec L_{i}$ is representable by a scheme; this will prove the result. For each $i$ we have two mutually exclusive cases.

\begin{enumeratea}

\item There exists an extension $L_{i}'$ of $L_{i}$ such that the pullbacks of $\pr_{1}^{*}\xi$ and $\pr_{2}^{*}\xi$ to $L_{i}'$ are isomorphic.

\item For any extension $L'_{i}$ of $L_{i}$,  the pullbacks of $\pr_{1}^{*}\xi$ and $\pr_{2}^{*}\xi$ to $L_{i}'$ are not isomorphic.
\end{enumeratea}

In the first case, the restriction $R_{L_{i}}$ is a $G$-torsor, and in the second it is empty. In both cases it is representable.
\end{proof}

\begin{proposition}\label{prop:fppf-gerbe}
Assume that $\xi$ is algebraic. Then $\cG_{\xi}$ is an fppf gerbe over $\spec k(\xi)$, where $k(\xi)$ is an intermediate extension $k \subseteq k(\xi) \subseteq K$, with $k(\xi)$ finite over $k$. Furthermore, there is a cartesian diagram
   \[
   \begin{tikzcd}
   \cB_{K}\underaut_{K}\xi \rar\dar & \cG_{\xi} \dar\\
   \spec K \rar & \spec k(\xi)\,.
   \end{tikzcd}
   \]
\end{proposition}

\begin{proof}
We can assume that $K$ is finite over $k$. Since $\spec K \arr \cG_{\xi}$ is flat and surjective, and every morphism to $\spec K$ is flat, it follows that every morphism to $\cG_{\xi}$ is flat. In particular, the inertia of $\cG_{\xi}$ is flat over $\cG_{\xi}$; by  \cite[\href{https://stacks.math.columbia.edu/tag/06QJ}{Proposition 06QJ}]{stacks-project} it follows that $\cG_{\xi}$ is an fppf gerbe over an algebraic space $Z$ over $\spec k$ (in the more general sense of \cite{stacks-project}). Under the present conditions it is immediate to check that the diagonal of $Z$ is of finite type; hence $Z$ has dense open subset that is a scheme  \cite[\href{https://stacks.math.columbia.edu/tag/06NH}{Proposition 06NH}]{stacks-project}. Since $\spec K \arr Z$ is flat and surjective, it follows immediately that $Z$ has to be the spectrum of a field, which proves the first statement.

For the second, the pullback $\spec K\times_{\spec k(\xi)}\cG_{\xi}$ is an fppf gerbe over $\spec K$. Since $\xi\colon \spec K \arr \cM$ factors, by definition, through $\cG_{\xi}$, we obtain a section $\spec K \arr \spec K\times_{\spec k(\xi)}\cG_{\xi}$; the automorphism group scheme of the corresponding object is $\underaut_{K}\xi$. This concludes the proof.
\end{proof}

\begin{definition}
The field $k(\xi)$ above is called the \emph{field of moduli} of $\xi$.
\end{definition}

The field of moduli has the following interpretation. Denote by $R$ the functor $\underisom_{K\otimes_{k}K}(\pr_{1}^{*}\xi, \pr_{2}^{*}\xi)$, which is an algebraic space of finite type over $K\otimes_{k}K$, by Proposition~\ref{prop:algebraic}.

\begin{proposition}\label{prop:description-field-moduli}
The field of moduli $k(\xi)$ is the equalizer of the two arrows $\pr_{1}^{*}$ and $\pr_{2}^{*}\colon K \arr \cO(R)$.
\end{proposition}

\begin{proof}
Since $\cG_{\xi}$ is a gerbe over $\spec k(\xi)$, we have $k(\xi) = \cO(\cG_{\xi})$. The morphism $\spec K \arr \cG_{\xi}$ is an fpqc cover, and $R = \spec K \times_{\cG_{\xi}} \spec K$, hence $\cO(\cG_{\xi})$ is the equalizer the two arrows in question.
\end{proof}

In the ``classical'' case, where $K/k$ is a (not necessarily finite) Galois extension we obtain the following interpretation of the field of moduli.

\begin{proposition}
Suppose that $K/k$ is a Galois extension with Galois group $G$. For each $\gamma \in G$ call $\gamma^{*}\xi$ the pullback of $\xi$ along $\gamma\colon \spec K \arr \spec K$; call $\overline{K}$ the algebraic closure of $K$. Let $H \subseteq G$ the subgroup consisting of elements $\gamma \in G$ such that $(\gamma^{*}\xi)_{\overline{K}}$ is isomorphic to $\xi_{\overline{K}}$. Then $k(\xi)$ is the field of invariants $K^{H}$.
\end{proposition}

\begin{proof}
First, assume that $K$ is finite over $k$. In this case $\spec(K \otimes_{k}K)$ is a disjoint union $\coprod_{g \in G} \spec K$. The image of the natural morphism $R \arr \spec(K \otimes_{k}K)$ is $\coprod_{g \in H} \spec K \subseteq \coprod_{g \in G} \spec K$; hence $\pr_{1}^{*}K$, $\pr_{2}^{*}\colon K \arr \cO(R)$ factor through the pullback $\prod_{g \in H} K \arr \cO(R)$, which is injective. The conclusion follows from Proposition~\ref{prop:description-field-moduli}.

If $K$ is not finite over $k$, choose an intermediate extension $k \subseteq K' \subseteq K$ such that $(X, \xi)$ descends to $K'$. The result for $K'$ is easily seen to imply that for $K$.
\end{proof}

\begin{definition}
The object $\xi$ is \emph{tame} if it is algebraic, and $\underaut_{K}\xi$ is finite and linearly reductive.
\end{definition}

Equivalently, the object $\xi$ is tame if $\cG_{\xi}$ is a tame stack.

\subsection{Residual gerbes and moduli spaces}

In case $\cM$ is an algebraic stack with finite inertia, there is another interesting interpretation of the field of moduli of a tame object.

Assume that $\cM$ is an algebraic stack with finite inertia, with moduli space $\cM \arr M$, and let $m \in M$ be a point. By definition of a moduli space there exists an object $\xi$ over the algebraic closure $\overline{k(m)}$; we say that the point $m$ is \emph{tame} if $\underaut_{\overline{k(m)}}\xi$ is linearly reductive.

\begin{definition}\label{def:residual-gerbe-point}
Assume that $m \in M$ is a tame point. The \emph{residual gerbe of $m$} is defined to be
   \[
   \cG_{m} \eqdef \bigl(\spec k(m)\times_{M}\cM\bigr)_{\mathrm red}\,.
   \]
\end{definition}

\begin{proposition}
The residual gerbe $\cG_{m}$ is a finite tame gerbe over $k(m)$.
\end{proposition}

\begin{proof}
By \cite[Proposition 3.6]{dan-olsson-vistoli1} the tame points of $M$ form an open subspace $M' \subseteq M$; we can replace $M$ with $M'$, and assume that $\cM$ is tame. Since formation of moduli spaces of tame stacks commutes with base change, we have that the moduli space of $\spec k(m)\times_{M}\cM$ is $\spec k(m)$; and from this, that the moduli space of $\cG_{m}$ is $\spec k(m)$. From \cite[Proposition~06RC]{stacks-project} it follows that $\cG_{m}$ is a gerbe over $k(m)$, as claimed.
\end{proof}

\begin{proposition}\label{prop:field-moduli-space}
Assume that $\cM$ is an algebraic stack with finite inertia locally of finite type over $k$, with moduli space $\cM \arr M$. Let $\xi\colon \spec K \arr \cM$ be a tame object, and call $m \in M$ the image of the composite $\spec K \xarr{\xi} \cM \arr M$. Then the residual gerbe of $\xi$ is $\cG_{m}$, and the field of moduli $k(\xi)$ is the residue field $k(m)$.
\end{proposition}

\begin{proof}
We may assume that $K$ is finite over $k$. The morphism $\spec K \arr \cM$ factors through $\spec k(m)\times_{M}\cM$; since $\spec K$ is reduced we get a factorization $\spec K \arr \cG_{m} \subseteq \cM$. Since $\cG_{m}$ is a finite gerbe, it follows that $\spec K \arr \cG_{m}$ is flat and finite, hence it is an fppf cover. The result follows from this.
\end{proof}

It is easy to give counterexamples to the statement of Proposition~\ref{prop:field-moduli-space} without the tameness hypothesis. The point is that the moduli space of $\spec k(m)\times_{M}\cM$ may be a non-trivial purely inseparable extension $k'$ of $k(m)$; and in this case the argument above shows that the field of moduli of $\xi$ is $k'$.

\subsection{The basic question}

Now assume that $\cM \arr \aff$ is an fppf locally finitely presented stack, $\xi\colon\spec \overline{k} \arr \cM$ an algebraic object defined over the algebraic closure of $k$, $k(\xi) \subseteq \overline{k}$ its field of moduli. Is $\xi$ defined over its field of moduli $k(\xi)$? This is equivalent to asking whether $\cG_{\xi}\bigl(k(\xi)\bigr) \neq \emptyset$. 

From now on we will consider objects $\xi$ defined over the algebraic closure $\overline{k}$ of $k$; from Proposition~\ref{prop:same-gerbe} it is clear that this is not a restriction.

\section{Application of nonabelian cohomology}\label{sec:nonabelian}

In the situation above, assume that $\xi \in \cM(\overline{k})$ is an algebraic object. 

As an immediate corollary of the fact that every affine gerbe over a finite field is neutral \cite[Theorem~8.1]{diproietto-tonini-zhang}, we get the following.

\begin{proposition}
Assume that $k$ is finite and $\underaut_{\overline{k}}\xi$ is affine. Then $\xi$ is defined over its field of moduli.
\end{proposition}

One can also apply standard results on the classification of gerbes, which usually go under the name of Grothendieck--Giraud nonabelian cohomology \cite{giraud}, to get conditions ensuring that $\xi$ is defined over its field of moduli. This works very cleanly when $\underaut_{\overline{k}}\xi$ is finite and reduced. 

\smallskip

Set $G \eqdef \underaut_{\overline{k}}\xi$, and assume for the rest of the section that $G$ is finite and reduced; according to our general conventions we think of $G$ as an ordinary group. Denote by $\aut G$ the group of automorphisms of $G$, and by $\out G$ its group of outer automorphisms, that is, the cokernel of the homomorphism $G \arr \aut G$ given by conjugation.

\begin{proposition}\label{prop:moduli1}
Assume that the following conditions are satisfied:

\begin{enumerate1}

\item the center of $G$ is trivial, and

\item the projection $\aut G \arr \out G$ is split.

\end{enumerate1}

Then $\xi$ is defined over its field of moduli.
\end{proposition}

This should be compared with \cite[Corollary~4.3(b)]{debes-emsalem}.

In similar spirit we get the following, which is a generalization of \cite[Corollary~4.3(a)]{debes-emsalem}, with the same proof. 

\begin{proposition}\label{prop:moduli2}
Assume that the absolute Galois group of $k$ has cohomological dimension at most $1$. Then $\xi$ is defined over its field of moduli.
\end{proposition}

These two propositions are immediate corollaries of the following standard application of nonabelian cohomology.

\begin{lemma}\label{lem:moduli}
Let $G$ be a finite group, $k$ a field. Let $\cG \arr \aff[k]$ be a gerbe such that $\cG_{\overline{k}}$ is isomorphic to $\cB_{\overline{k}}G$. Assume that either

\begin{enumerate1}

\item $G$ has trivial center, and $\aut G \arr \out G$ is split, or

\item $k$ has cohomological dimension~$1$.

\end{enumerate1}

Then $\cG$ is neutral.
\end{lemma}

\begin{proof}
	Denote by $k^{s}$ the separable closure of $k$ and by $\Gamma=\operatorname{Gal}(k^{s}/k)$ be the absolute Galois group of $k$, we have that $\cG_{k^{s}}$ is isomorphic to $\cB_{k^{s}}G$ by the following Lemma~\ref{lem:closed-neutral}. The tautological section $\spec k^{s}\to \cB_{k^{s}}G\to \cG$ induces a continuous homomorphism $\Gamma\to\out G$. Under both conditions (1) and (2), we have a lifting $\Gamma\to\aut G$: for (1) this is obvious, while for (2) this follows from the fact that $G$ is projective \cite[Proposition~45]{serre}. By descent theory, the homomorphism $\Gamma\to\aut G$ induces a finite étale group scheme $\underline{G}$ which is a twisted form of $G$.
	
	Let $L$ be the non-abelian band, in the sense of \cite{giraud}, of $\cG$. By construction, $L$ is represented by $\underline{G}$. By \cite[Théorème 3.3.3]{giraud}, under both conditions (1) and (2), the gerbe $\cG$ is the only gerbe banded by $L$. Since the classifying stack $\cB_{k}\underline{G}$ is banded by $L$ by construction, we get that $\cG\simeq \cB_{k}\underline{G}$, i.e. $\cG$ is neutral.
\end{proof}

\begin{lemma}\label{lem:closed-neutral}
	A finite gerbe with unramified diagonal over a separably closed field is neutral.
\end{lemma}

\begin{proof}
	Let $k$ be a separably closed field with algebraic closure $\overline{k}$, and $\cG$ a finite gerbe over $k$ with unramified diagonal. There exists a finite extension $k'/k$ with a section $s\in\cG(k')$. The scheme $\underisom(p_{1}^{*}s,p_{2}^{*}s)$ is finite étale over the artinian local ring $k'\otimes_{k}k'$ and we have a lifting $\spec k'\to\underisom(p_{1}^{*}s,p_{2}^{*}s)$ of the closed point $\spec k'\subset\spec k'\otimes_{k}k'$, hence we get a section $\spec k'\otimes_{k}k'\to\underisom(p_{1}^{*}s,p_{2}^{*}s)$. By descent theory we obtain that $s$ descends to $k$.
\end{proof}

These results do not apply in many cases of great interest, for example, when $G$ is abelian and $k$ is a number field. 

Our main theorem, which is a generalization of \cite[Corollary~4.3(c)]{debes-emsalem}, gives a criterion for this to happen, when $\xi$ is tame, for an interesting class of stacks, whose objects are algebraic spaces with additional structure.

\section{The main theorem}\label{sec:main}

\subsection{Categories of structured spaces} We will denote by $\as$ the fibered category over $\aff$ whose objects over an affine scheme $S$ over $k$ are flat finitely presented morphisms $X \arr S$, where $X$ is an algebraic space.

\begin{definition}
A \emph{category of structured spaces} over $k$ is a locally finitely presented fppf stack $\cM \arr \aff$, with a faithful cartesian functor $\cM \arr \as$.

\end{definition}

Examples spring to mind.

\begin{examples}\hfil
\begin{enumerate1}

\item The category $\as$ itself is a category of structured spaces.

\item Any condition that we impose on objects $X \to S$ in $\as$ that is stable under base change and fppf local defines a full subcategory $\cM \subseteq \as$ that is a category of structured spaces. Thus, for example, the stacks $\cM_{g}$ and $\overline{\cM}_{g}$ of smooth, or stable, curves of genus $g$, the stack of smooth abelian varieties, of projective surfaces, and so on.

\item The stacks $\cM_{g,n}$ and $\overline{\cM}_{g,n}$ of $n$-pointed smooth, or stable, curves of genus~$g$ are all categories of structured spaces.

\item The category of smooth polarized projective schemes is a category of structured spaces.

\end{enumerate1}
\end{examples}

An object $\xi$ of a category of structured spaces $\cM$ will be denoted by $(X \to S, \xi)$, or simply $(X,\xi)$, where $X \arr S$ is the image of $\xi$ in $\as$; we want to think of $(X, \xi)$ as an algebraic space with an additional structure.

A category of structured spaces $\cM$ has a \emph{universal family} $\cX \arr \aff$; the objects of $\cX$ are triples $(X \to S, \xi, x)$, where $(X \arr S, \xi)$ is an object of $\cM(S)$, and $x\colon S \arr X$ is a section of the morphism $X \arr S$. We have an obvious morphism $\cX \arr \cM$ which forgets the section. If $S \arr \cM$ is a morphism, corresponding to an object $(X \to S, \xi)$ of $\cM(S)$, then the fibered product $S \times_{\cM} \cX$ is equivalent to $X$; hence the morphism $\cX \arr \cM$ is representable, flat and finitely presented. The universal family is itself a category of structured spaces.

Now, let $(X, \xi)$ be an algebraic object of $\cM(\overline{k})$, and denote by $\cX_{(X,\xi)}$ the fibered product $\cG_{(X, \xi)}\times_{\cM}\cX$. We have a commutative diagram
   \[
   \begin{tikzcd}
   X\rar\dar\ar[rd, phantom, "\square"]& {[X/\underaut_{\overline{k}}\xi]}\ar[rd, phantom, "\square"] \rar\dar
   & \cX_{(X, \xi)}\ar[rd, phantom, "\square"] \rar \dar & \cX \dar\\
   \spec \overline{k}\rar\ar[dr, equal]&\cB_{\overline{k}}\underaut_{\overline{k}}\xi \dar\rar\ar[rd, phantom, "\square"] &
   \cG_{(X, \xi)} \dar\rar & \cM\dar\\
   &\spec \overline{k} \rar & \spec k(X, \xi) \rar & \spec k
   \end{tikzcd}
   \]
in which the squares marked with $\square$ are cartesian.

\begin{definition}
	Assume that $\cG_{(X,\xi)}$ (and hence $\cX_{(X,\xi)}$) has finite inertia. The \emph{compression} of $(X,\xi)$ is the coarse moduli space of $\cX_{(X,\xi)}$, it is an algebraic space over the field of moduli $k(X,\xi)$. We denote the compression using bold letters, for instance we write $\bX_{(X,\xi)}$ for the compression of $(X,\xi)$.
\end{definition}

Since formation of moduli spaces commutes with flat base change (see \cite{conrad}) we have
   \[
   \spec \overline{k} \times_{\spec k} \bX_{(X, \xi)} = X/\underaut_{\overline{k}}\xi\,.
   \]
In other words, while $X$ does not necessarily descend to $k(X, \xi)$, the quotient $X/\underaut_{\overline{k}}\xi$ always does, in a canonical fashion. This is a more general version of \cite[Theorem 3.1]{debes-emsalem}.

\subsection{The main result}

\begin{theorem}\label{thm:moduli}
Let $\cM \arr \aff$ be a category of structured spaces, $(X, \xi)\in \cM(\overline{k})$ a tame object with $X$ integral.

Assume that there exists a dominant rational map $Y \dashrightarrow \bX_{(X, \xi)}$ where $Y$ is an integral algebraic space of finite type over $k(X,\xi)$ with a $k(X,\xi)$-rational regular point. Then $(X,\xi)$ is defined over its field of moduli $k(X, \xi)$.
\end{theorem}

In this proof, and in the rest of the paper, a crucial role is played by the Lang--Nishimura theorem for tame stack that we prove in \cite{giulio-angelo-valuative}. For the convenience of the reader we recall its statement.

\begin{theorem}[{\cite[Theorem 4.1]{giulio-angelo-valuative}}]\label{LN}
	Let $S$ be a scheme and $X\dashrightarrow Y$ a rational map of algebraic stacks over $S$, with $X$ locally noetherian and integral and $Y$ tame and proper over $S$. Let $k$ be a field, $s\colon \spec k \arr S$ a morphism. Assume that $s$ lifts to a regular point $\spec k \arr X$; then it also lifts to a morphism $\spec k \arr Y$.
\end{theorem}

\begin{proof}[Proof of Theorem~\ref{thm:moduli}]
	There exists an $\underaut_{\overline{k}}\xi$-invariant open subset $U\subset X$ such that the action of $\underaut_{\overline{k}}\xi$ on $U$ is free: by hypothesis the action of $\underaut_{\overline{k}}\xi$ on $X$ is faithful, and for each non-trivial subgroup $G\subseteq \underaut_{\overline{k}}\xi$ the locus of points of $X$ fixed by $G$ is a proper closed subset of $X$. This in turn implies that there exists an open substack $\cU\subset\cX_{(X, \xi)}$ which is an algebraic space; then the composite $\cU \subseteq \cX_{(X, \xi)} \arr \bX_{(X, \xi)}$ is an open embedding.
	
	Because of this, the hypothesis gives us a rational map $Y\dashrightarrow \cU$. We conclude by applying Theorem~\ref{LN} to the composite $Y\dashrightarrow \cU\subset\cX_{(X, \xi)}\to \cG_{(X, \xi)}$.
\end{proof}

Note that if $X$ is smooth of dimension $1$ over $\overline{k}$, then $\bX_{(X,\xi)}$ is also smooth over $k(X,\xi)$; hence in this case we get the following.

\begin{corollary}\label{cor:moduli}
Let $\cM \arr \aff$ be a category of structured spaces, $(X, \xi)\in \cM(\overline{k})$ a tame object such that $X$ is integral, smooth and $1$-dimensional. Assume that the compression $\bX_{(X,\xi)}$ has a $k(X,\xi)$-rational point.  Then $(X,\xi)$ is defined over its field of moduli $k(X, \xi)$.
\end{corollary}

When $(X,\xi)$ is a smooth projective curve with no additional structure, this is \cite[Corollary~4.3(c)]{debes-emsalem}.

\subsection{The case of pointed spaces}\label{ssec:pointed}

One case in which we can ensure the existence of a rational point on $\bX_{(X, \xi)}$ is the case of pointed spaces. For the rest of the paper we will use the following definition.

\begin{definition}
A \emph{pointed space} $(X, p)$ over an affine scheme $S$ over $k$ is a flat locally finitely presented morphism $X \arr S$ with a section $p\colon S \arr X$ landing in the smooth locus of $f$.
\end{definition}

Pointed spaces form a fibered category $\pas \arr \aff$; there is an obvious representable cartesian functor $\pas \arr \as$, which forgets the section.

\begin{definition}
A \emph{category of pointed structured spaces} over $k$ is a locally finitely presented fppf stack $\cM \arr \aff$, with a faithful cartesian functor $\cM \arr \pas$.
\end{definition}

If $\cM \arr \pas$ is a category of pointed structured spaces and $S$ is a scheme, an element of $\cM(S)$ will be denoted by $(X, p, \xi)$, where $(X,\xi)$ is the corresponding structured space given by the composition $\cM \arr \pas \arr \as$ and $p\colon S \arr X$ is the given section.

If $\cM$ is a category of pointed structured spaces, it can be considered as a category of structured spaces by composing $\cM \arr \pas$ with the forgetful morphism $\pas \arr \as$. One can think of categories of pointed structured spaces as a category of structured spaces, in which the structure includes a smooth marked point.

There are many natural examples of categories of pointed structured spaces: for example, the category of abelian varieties, or the category of $n$-pointed stable, or smooth, curves, for $n \geq 1$.

\begin{lemma}\label{lem:pointed}
	Let $\cM$ be a category of pointed structured spaces, $(X, p, \xi) \in \cM(\overline{k})$ a tame object. The compression $\bX_{(X, p, \xi)}$ has a rational point $\bfp$ over $k(X, p, \xi)$ such that $\bfp_{\overline{k}}$ corresponds to $p$ via the identification $\bX_{(X, p, \xi),\overline{k}} = X/\underaut(X, p, \xi)$. 
\end{lemma}

\begin{proof}
	Consider the universal family $\cX \arr \cM$, and its smooth locus $\cX_{\rm sm} \subseteq \cX$ (this is the largest open substack of $\cX$ where the restriction of $\cX \arr \cM$ is smooth). The fibered product $\cM\times_{\as}\pas$ is canonically isomorphic to $\cX_{\rm sm}$; hence the cartesian functor $\cM \arr \pas$ induces a section $\cM \arr \cX_{\rm sm}$ of the projection $\cX_{\rm sm} \subseteq \cX \arr \cM$. By restricting to $\cG_{(X, p, \xi)}$ we obtain a section $\cG_{(X, p, \xi)} \arr \cX_{(X, p, \xi)}$ of the projection $\cX_{(X, p, \xi)} \arr \cG_{(X, p, \xi)}$, and, passing to moduli spaces, a section $\spec k(X, p, \xi) \arr \bX_{(X, p, \xi)}$ of the projection $\bX_{(X, p, \xi)} \arr \spec k(X, p, \xi)$, or, in other words, a $k(X, p, \xi)$-rational point of $\bX_{(X, p, \xi)}$.
\end{proof}



When $X$ is $1$-dimensional, it follows that $\bfp\in\bX_{(X, p, \xi)}(k(X, p, \xi))$ is smooth. Hence we get the following.

\begin{corollary}\label{cor:moduli-dim1}
Let $\cM$ be a category of pointed structured spaces, $(X,p,\xi)$ a tame object of $\cM(\overline{k})$, such that $X$ is $1$-dimensional and integral. Then $(X, p,\xi)$ is defined over its field of moduli.
\end{corollary}

As a consequence, we recover the following result by Dèbes--Emsalem.

\begin{corollary}[{\cite[Corollary 5.4]{debes-emsalem}}]
Let $g\ge 1$ and $n\ge 1$ be positive integers and $k$ a field. Every smooth $n$-pointed curve of genus $g$ over $\overline{k}$ with tame automorphism group scheme is defined over its field of moduli.


In particular, if $\cha k=0$, $K/k$ is any extension and $\rM_{g,n}$ is the coarse moduli space of smooth $n$-pointed curves of genus~$g$ over $k$, every $K$-valued point of $\rM_{g,n}$ comes from a smooth $n$-pointed curve of genus~$g$ over $K$.
\end{corollary}

\begin{proof}
Let $\cM_{g,n}$ be the stack of $n$-pointed curves of genus~$g$ over $k$, since $n\ge 1$ we may think of it as a category of pointed structured spaces, the first part then follows from Corollary~\ref{cor:moduli-dim1}. Now assume that $\cha k=0$, let $K/k$ be any extension and $\spec K\to \rM_{g,n}$ a point. Let $X$ be an $n$-pointed smooth curve over $\overline{K}$ corresponding to the composite $\spec \overline{K} \arr \spec K \arr \rM_{g,n}$; by Proposition~\ref{prop:field-moduli-space} we have that the field of moduli of $X$ is $K$, hence $X$ is defined over $K$.
\end{proof}

This fails for stable curves that are not irreducible: one can give examples of $1$-pointed stable curves that are not defined over their field of moduli. The issue here is that the automorphism group of the pointed curve may not act faithfully on the component containing the marked point.

Corollary~\ref{cor:moduli-dim1} fails in dimension higher than $1$. There are many counterexamples. For example, fix a positive integer~$g$, let $\rA_{g}$ the moduli space of principally polarized $g$-dimensional abelian varieties over $\CC$, and let $k$ be its field of rational functions. Call $X$ the corresponding abelian variety over the algebraic closure $\overline{k}$, which we can think of as $1$-pointed variety $(X, 0)$. By Proposition~\ref{prop:field-moduli-space} the field of moduli of $(X,0)$ is $k$. As we mentioned in the introduction, G.~Shimura showed in \cite{shimura-field-rationality-abelian} that when $k  = \CC$ $(X,0)$ is defined over $k$ if and only if $g$ is odd (see \cite[Appendix]{brosnan-reichstein-vistoli3} for a refinement of this statement due to Najmuddin Fakhruddin). In this case the group of automorphisms is cyclic of order~$2$.

Given a positive integer $d$, we will study a natural class of discrete finite groups over $k$, with the property that if $(X,\xi)$ is a tame object of $\cM$, and $\aut_{\overline{k}}\xi$ is in this class, then $\xi$ is defined over its field of moduli.

\section{The arithmetic of tame quotient singularities}\label{sec:arithmetic}

Let $S$ be an algebraic space of finite type over a field $k$. We say that $S$ \emph{has tame quotient singularities} if there is an étale cover $\{S_{i} \arr S\}$ and, for each $i$, a smooth algebraic space $U_{i}$ and a finite group $G_{i}$ of order not divisible by $\cha k$ acting on $U_{i}$, such that $S_{i}$ is isomorphic to $U_{i}/G$. In particular, $S$ is normal.

More generally, one could consider spaces that are étale-locally quotients of smooth algebraic spaces by finite linearly reductive group schemes, as in \cite{satriano-chevalley-shephard-todd}; but the technology to adequately deal with these in our context still does not seem to be completely in place, which forces us to limit ourselves to considering tame Deligne--Mumford stacks, as opposed to general tame stacks.

\subsection{Minimal stacks} The following is known, see \cite[Proposition 2.8]{angelo-intersection}. Our statement is slightly different from the one in the reference, though, so we give details.

\begin{proposition}\label{prop:min-stack}
The moduli space of a smooth tame Deligne--Mumford stack with finite inertia over $k$ has tame quotient singularities.

Conversely, if $S$ is an algebraic space with finite quotient singularities, there exists a smooth tame Deligne--Mumford stack with finite inertia $\widehat{S}$ with moduli space $S$, with the property that the morphism $\widehat{S} \arr S$ is an isomorphism over the smooth locus of $S$.

Furthermore, if $V$ is a smooth integral Deligne--Mumford stack with a dominant morphism $V\to \widehat{S}$, there exists a factorization $V\to\widehat{S}\to S$, unique up to a unique isomorphism. In particular, $\widehat{S}$ is unique, up to a unique isomorphism.
\end{proposition}

\begin{proof}
	Let $\cX$ be a smooth, tame Deligne--Mumford stack with finite inertia, and let $M$ be its moduli space, we want to show that $M$ has tame quotient singularities. By \cite[Lemma~2.2.3]{dan-vistoli02}, we may assume that $\cX = [U/G]$, where $G$ is a finite group, so that $M = X/G$. If $u_{0}\colon \spec \Omega \arr U$ is a geometric point of $U$, and $G_{u_{0}}$ is the stabilizer of $u_{0}$, then the natural morphism $U/G_{u_{0}} \arr U/G$ is étale in neighborhood of $u_{0}$; but $G_{u_{0}}$ has order prime to $\cha k$, because $[U/G]$ is tame, so the result follows.
	
	If $S$ has tame quotient singularities, then the second half of the proof of \cite[Proposition 2.8]{angelo-intersection} shows the existence of a stack $\widehat{S}\to S$ as in the statement (the reference's assumption $\cha k=0$ is not used in the relevant part of the proof).
	
	Now let $V\to\widehat{S}$ be as in the statement. Since $S$ is normal, $\widehat{S}\to S$ is an isomorphism in codimension $1$. Let $U$ be the normalization of $V\times_{S}\widehat{S}$, since everything is of finite type over $k$ and $S$ is the moduli space of $\widehat{S}$, then $U\to V$ is proper and birational. By purity of branch locus, it is étale too, hence $U\simeq V$ and we obtain the desired morphism $V\to\widehat{S}$. The uniqueness follows from the fact that $\widehat{S}$ is separated.
\end{proof}

We call the stack $\widehat{S}$ above the \emph{minimal stack} of $S$ (also called the \emph{canonical stack} in the literature). Clearly, if $S$ has tame quotient singularities, and $k'$ is an extension of $k$, the space $S_{k'}$ obtained by base change also has tame quotient singularities, and the minimal stack of $S_{k'}$ is $\widehat{S}_{k'}$. In other words, formation of the minimal stack commutes with extensions of the base field. Furthermore, if $S' \arr S$ is an étale morphism and $S$ has tame quotient singularities, then so does $S'$, and $\widehat{S'} = S'\times_{S}\widehat{S}$.

It is known that algebraic spaces with tame quotient singularities have a resolution of singularities, that is, there is a proper birational morphism $\widetilde{S} \arr S$ where $\widetilde{S}$ is a smooth algebraic space over $k$. If $k$ is perfect, this is \cite[Theorem~E]{bergh-rydh-functorial-destackification}; in the general case it is obtained by applying \cite[Theorem~B]{bergh-rydh-functorial-destackification} to the minimal stack $\widehat{S} \arr S$.

\subsection{Singularities and fundamental gerbes} Let $k$ be a field; a \emph{tame quotient singularity} over $k$ is a pair $(S, s)$, where $S$ is a an integral scheme of finite type over $k$ with tame quotient singularities and $s \in S(k)$ is a $k$-rational point. No other kinds of singularities will appear in this paper, so from now on a tame quotient singularity will be called simply \emph{a singularity}.

Two singularities $(S, s)$ and $(S', s')$ are \emph{equivalent} if there exists a singularity $(S'', s'')$, together with étale maps $S'' \arr S$ and $S'' \arr S'$ sending $s''$ into $s$ and $s'$ respectively. This is true if and only if the complete local $k$-algebras $\widehat{\cO}_{S,s}$ and $\widehat{\cO}_{S',s'}$ are isomorphic \cite[Corollary~2.6]{artin-approximation}.

\begin{definition}
	Given a singularity $(S,s)$, the \emph{fundamental gerbe} $\cG_{(S,s)}$ of $(S,s)$ is the residual gerbe, as in Definition~\ref{def:residual-gerbe-point}, of $\widehat{S}\to S$ at $s$. The \emph{fundamental group} $G_{(S,s)}$ of $(S,s)$ is the automorphism group of any geometric point of $\cG_{(S,s)}$.
\end{definition}

Thus, by definition, the fundamental group of $(S,s)$ is a finite group, or order prime to $\cha k$. It is well defined up to a non-canonical isomorphism. 

One can prove that $\cG_{(S,s)}$ is the local fundamental gerbe of $S$, in the following sense. Let $S' \eqdef \spec\cO_{S,s}^{\rm h}$, where $\spec\cO_{S,s}^{\rm h}$ is the henselization of $\cO_{S,s}$, and let $U \subseteq S'$ be the smooth locus of $S$. Set $\widehat{S'} \eqdef S' \times_{S}\widehat{S}$. Then $\cG_{(S,s)}$ is the fundamental gerbe, in the sense of \cite{borne-vistoli} of $U$ and also of $\widehat{S'}$. We do not prove this here, as it is not needed in what follows.

Since formation of the minimal stack commutes with étale morphisms, equivalent singularities have isomorphic fundamental gerbes.

\begin{lemma}\label{lem:tqs-closed}
	Assume that $k$ is separably closed. Let $\cX$ be a smooth, tame Deligne--Mumford stack which is generically a scheme, with moduli space $\cX \arr S$. Let $\xi$ be an object in $\cX(k)$ and $s\in S(k)$ its image of $\xi$. 
	
	There exists a faithful representation $\aut\xi\subset\GL_{d}(k)$ such that $(S,s)$ is equivalent to $(\AA^{d}/\aut\xi,[0])$, and the quotient of $\aut\xi$ by the subgroup generated by pseudoreflexions is isomorphic to the fundamental group of $(S,s)$.
\end{lemma}

\begin{proof}
	After passing to an étale neighborhood of $s\in S$, we may assume $\cX\simeq [U/H]$ with $U$ smooth and $H$ finite of order prime to $\cha k$. Since $k$ is separably closed, the rational point $\xi\in \cX(k)$ lifts to a rational point $u\in U(k)$. Let $H_{u}\subset H$ be the stabilizer of $u$, then $U/H_{u}\to U/H\simeq S$ is étale in $[u]$, hence we may replace $\cX, H$ with $[U/H_{u}], H_{u}=\aut\xi$ and assume that $H=\aut\xi$ and that $u$ is a fixed point.
	
	Call $V$ the tangent space of $U$ at $u$; then $\aut\xi$ acts on $V$, and by fixing a basis we get a representation $\aut\xi \arr \GL_{d}(k)$. By Cartan's lemma, after passing to an equivariant étale neighborhood of $u$ in $U$ we may assume that there exists an étale $\aut\xi$-equivariant map $U \arr V$. This implies that the action of $\aut\xi$ on $V$ is faithful, and that $(S,s) = (U/\aut\xi, [u])$ is equivalent to $(V/\aut\xi, [0])$.
	
	Denote by $P\subset\aut\xi$ the subgroup generated by pseudoreflexions, by the Chevalley--Shephard--Todd theorem $V/P$ is smooth and $\aut\xi/P$ acts on it with with no fixed points in codimension $1$. It follows that $[(V/P)/(\aut\xi/P)]$ is the minimal stack of $V/\aut\xi$, and hence $\aut\xi/P$ is the fundamental group of $(V/\aut\xi, [0])\sim(S,s)$.
\end{proof}

\begin{corollary}\label{cor:tqs-rep}
	Assume that $k$ is separably closed, and let $(S,s)$ be a tame quotient singularity over $k$. There exists a faithful representation $G_{(S,s)}\subset\GL_{d}(k)$ with no pseudoreflexions such that $(S,s)\sim (\AA^{d}_{k}/G_{(S,s)},[0])$.
\end{corollary}

\begin{proof}
	Thanks to Lemma~\ref{lem:closed-neutral}, we may apply Lemma~\ref{lem:tqs-closed} to $\widehat{S}$.
\end{proof}

\subsection{Liftable singularities}

The following is immediate from Theorem~\ref{LN}.

\begin{proposition}\label{prop:char-liftable}
Let $(S,s)$ be a tame quotient singularity, $\widetilde{S} \arr S$ a resolution of singularities. The following conditions are equivalent.

\begin{enumerate1}

\item $\widetilde{S}$ has a $k$-rational point over $s$.

\item The fundamental gerbe $\cG_{(S,s)}$ of $(S,s)$ is neutral.

\item The minimal stack $\widehat{S}\arr S$ has a $k$-rational point over $s$.

\item If $S' \arr S$ is a proper birational morphism, where $S'$ is an integral tame Deligne--Mumford stack, then $S'$ has a $k$-rational point over $s$.

\end{enumerate1}
\end{proposition}

\begin{definition}
Let $(S,s)$ be a tame quotient singularity. We say that $(S, s)$ \emph{liftable} if it satisfies the equivalent conditions of Proposition~\ref{prop:char-liftable}.
\end{definition}

\begin{remarks}\hfil
\begin{enumerate1}

\item If two singularities are equivalent, then one is liftable if and only if the other is.

\item Since formation of $\widehat{S}$ commutes with extension of the base field, we see that if $(S,s)$ is liftable over $k$, and $k'$ is an extension of $k$, then $(S, s)_{k'}$ is also liftable.

\end{enumerate1}
\end{remarks}

The following construction gives a criterion for a singularity to be liftable, which will be a fundamental tool in the rest of the paper.

\subsection{The blowup construction} Let $K/k$ be a separable closure. Let $(S,s)$ be a tame quotient singularity over $k$; by Corollary~\ref{cor:tqs-rep} $(S_{K},s_{K})$ is equivalent to $(\AA^{d}/G,[0])$ for some $G\subset\GL_{d}(K)$, such that $G$ contains no pseudoreflexions and has order prime to $\cha k$. Denote by $\overline{G}$ the image of $G$ in $\PGL_{d}(K)$.

Denote by $\cG$ the fundamental gerbe $\cG_{(S,s)}$, and by $\cN$ the normal bundle of $\cG$ in $\widehat{S}$; this is a vector bundle over $\cG$ of rank~$d$. Also, denote by $\cB$ the blowup of $\widehat{S}$ along $\cG$; clearly the exceptional divisor $\cE$ equals $\PP(\cN)$. Denote by $E$ the moduli space of $\cE$; the morphism $\cE \arr E$ factors through the minimal stack $\widehat{E}$. 

We have $\cE_{K} = [\PP^{d-1}_{K}/G]$, $\cN_{K} = [\AA^{d}_{K}/G]$, and $E_{K} = \PP^{d-1}_{K}/\overline{G}$.

\begin{definition}
	We say that $E$ is the \emph{associated variety} of the singularity.
\end{definition}

Recall that $\cB$ is the blow up of $\cG$ in $\widehat{S}$, let $B\to S$ be its coarse moduli space. Since $\cB$ is a tame stack and formation of coarse moduli spaces commutes with base change for tame stacks, the reduced fiber of $B\to S$ over $s$ is $E$.

\begin{corollary}\label{cor:points-blowup}
	There exists a rational morphism $E\dashrightarrow\cG$.
\end{corollary}

\begin{proof}
	The morphism $\cB\to B$ is birational. Since $\cB$ is smooth, then $B$ is normal, in particular it is regular at the generic point of $E$. It follows that the morphism $\spec k(E)\to B$ lifts to a morphism $\spec k(E)\to\cB$ by our version of the Lang--Nishimura theorem. Clearly, the composite $\spec k(E)\to\cB\to\widehat{S}$  factors through $\cG\to\widehat{S}$.
\end{proof}

Given an irreducible algebraic space $X$ of finite type over a field $k$, we say that a field extension $k'/k$ \emph{splits} $X$ if there exists a dominant rational map $Y\dashrightarrow X_{k'}$ where $Y$ is an integral scheme of finite type over $k'$ with a smooth $k'$-rational point. If $k$ splits $X$, then we say that $X$ is split. If $X$ has tame quotient singularities, using our version of the Lang--Nishimura theorem we see that $k'$ splits $X$ if and only if the minimal stack $\widehat{X}$ has a $k'$-rational point.

\begin{proposition}\label{prop:var-split}
	A field extension $k'/k$ splits $E$ if and only if it splits $\cG$.
\end{proposition}

\begin{proof}
	If $k'$ splits $\cG$, then there exists a morphism of $k$-stacks $\spec k' \arr \cG$; since $\cE$ is the projectivization of a vector bundle on $\cG$, this lifts to $\spec k' \arr \cE$; but $\cE$ maps to $\widehat{E}$, so $k'$ splits $E$.
	
If $k'$ splits $E$, since there exists a rational map $E\dashrightarrow\cG$ our version of the Lang--Nishimura theorem implies that $k'$ splits $\cG$.
\end{proof}

\subsection{$R$-singularities}

If $k$ and $k'$ are fields with the same characteristic, $(S,s)$ is a singularity over $k$ and $(S',s')$ is a singularity over $k'$, we say that $(S,s)$ and $(S',s')$ are \emph{stably equivalent} if there exists a common extension $k\subseteq K$ and $k'\subseteq K$, such that $(S,s)_{K}$ and $(S',s')_{K}$ are equivalent. It is easily checked that this is an equivalence relation on tame quotient singularities.

\begin{definition}
An \emph{$R$-singularity} is a tame quotient singularity, such that every singularity that is stably equivalent to it is liftable.
\end{definition}

From the definition, it is not clear that there are any non-trivial examples of $R$-singularities.

\subsection{$R_{d}$ groups}

\begin{definition}\label{def:Rd}

Let $d$ be a positive integer, $p$ be either $0$ or a prime, and $G$ be a finite group whose order is not divisible by $p$. We say that $G$ is an \emph{$R_{d}$ group in characteristic~$p$} if for every field $K$ of characteristic~$p$ and every faithful $d$-dimensional representation $G \subseteq \GL_{d}(K)$, the singularity $(\AA_{K}^{d}/G, [0])$ is an $R$-singularity.

If $G$ is $R_{d}$ in all characteristics not dividing the order of $G$, we say that $G$ is \emph{an $R_{d}$ group}, or simply that $G$ is $R_{d}$.
\end{definition}

Another way of stating this is the following. 

\begin{definition}\label{def:R-singularity}
Let $G$ be a finite group, $(S,s)$ a singularity over a field $k$ whose characteristic does not divide the order of $G$. We say that $(S,s)$ is a \emph{$G$-singularity} if there exists a field $K$ and a faithful $d$-dimensional representation $G \subseteq \GL_{d}(K)$ such that $(S,s)$ is stably equivalent to $\AA^{d}_{K}/G$.
\end{definition}

Then $G$ is an $R_{d}$ group if and only if every $G$-singularity is liftable.

\begin{remarks}\hfil
\begin{enumerate1}

\item As a point of terminology, we notice that if $G$ is not a subgroup of $\GL_{d}(K)$ for any field $K$ (for example, if $G$ contains an abelian subgroup of rank larger than~$d$), then it is vacuously an $R_{d}$ group. Such a group can not act faithfully on a $d$-dimensional variety with a smooth fixed point, so it will not actually appear in the statement of Theorem~\ref{thm:moduli2}. So, we are actually only interested in $R_{d}$ groups that are subgroups of some $\GL_{d}(K)$.

\item Every finite group is trivially $R_{1}$.

\item In the definition of an $R_{d}$ group, we may assume that $K$ is algebraically closed.

\end{enumerate1}
\end{remarks}

\begin{lemma}\label{lem:Rd}
Let $\cX$ be a geometrically integral tame Deligne--Mumford stack of dimension $d$ over a field $k$ with finite inertia and moduli space $\cX \arr M$. Assume that $\cX \arr M$ is a birational isomorphism. Let $\xi\in\cX(\overline{k})$ be a smooth geometric point with image $p\in M$.

If the automorphism group $\aut_{\overline{k}}\xi$ is an $R_{d}$ group in $\cha k$, then $(M, p)$ is an $R$-singularity. In particular, $p\in M$ lifts to a $k(p)$-rational point of $\cX$.
\end{lemma}

\begin{proof}
This is a direct consequence of Lemma~\ref{lem:tqs-closed}.
\end{proof}

The point of this definition is the following result. Let us put ourselves in the situation of subsection~\ref{ssec:pointed}: $\cM \arr \aff$ is a category of pointed structured spaces, $(X, p, \xi)\in \cM(\overline{k})$ an algebraic object, as in Theorem~\ref{thm:moduli}. 

\begin{theorem}\label{thm:moduli2}
Let $\cM \arr \aff$ be a category of pointed structured spaces, $(X, p, \xi)\in \cM(\overline{k})$ an algebraic object with $X$ integral of dimension~$d$. Assume that the automorphism group scheme $\underaut_{\overline{k}}(X,p)$ is finite, tame and reduced. If $\underaut_{\overline{k}}(X,p, \xi)$ is an $R_{d}$ group, then $(X,p, \xi)$ is defined over its field of moduli.
\end{theorem}

\begin{proof}
We apply Lemma~\ref{lem:Rd} to the stack $\cX_{(X,p,\xi)}$, with moduli space $\bX_{(X,p, \xi)}$, to conclude that the space $\bX_{(X,p, \xi)}$ has an $R$-singularity at the rational point corresponding to $p$. Then if $Z \arr \bX_{(X,p, \xi)}$ is a resolution of singularities we see that $Z$ has a regular rational point, so the conclusion follows from Theorem~\ref{thm:moduli}.
\end{proof}

We still have to give meaningful examples of $R_{d}$ groups. 

A fairly trivial class of $R_{d}$ groups is the following: say that a group is \emph{strongly $R_{d}$} if for any embedding $G \subseteq \GL_{d}(K)$, where $K$ is an algebraically closed field of characteristic not dividing $|G|$, we have that $G$ is generated by pseudoreflexions in $\GL_{d}(K)$. By the Chevalley--Shephard--Todd theorem we have that $\AA^{d}/G$ is smooth, hence every strongly $R_{d}$ group is also $R_{d}$.

The following is straightforward.

\begin{proposition}\hfil
\begin{enumerate1}

\item If $m$ is a positive integer, $\rC_{m}^{d}$ is strongly $R_{d}$.

\item Dihedral groups are strongly $R_{2}$.

\end{enumerate1}
\end{proposition}

To give examples of finite groups that are $R_{d}$ without being strongly $R_{d}$ is more complicated, and requires much more technology. Here are two results in this direction.

\begin{theorem}\label{thm:Rd-1}
Let $G$ a finite group with the following properties.

\begin{enumerate1}

\item The center of $G$ is trivial.

\item The projection $\aut G \arr \out G$ is split.

\item Either $G$ is perfect, or all proper normal subgroups of $G$ are perfect.

\end{enumerate1}

Then $G$ is $R_{d}$ for all $d$.
\end{theorem}

For example, these conditions are satisfied for all symmetric groups $\rS_{n}$ and alternating groups $\rA_{n}$ with $n \geq 5$, $n \neq 6$.  Also, there are infinitely many classes of simple groups such that the projection $\aut G \arr \out G$ is split; a complete classification is given in \cite{lucchini-menegazzo-morigi}.

While this result is interesting, it does not give any new examples of applications of Theorem~\ref{thm:moduli2}, because of Proposition~\ref{prop:moduli1}. The next result, however, does yield new examples.

\begin{theorem}\label{thm:Rd-2}
A group of order prime to $d!$ is $R_{d}$.
\end{theorem}

The following easy result allows us to give more examples of $R_{d}$ groups.

\begin{proposition}
Let $G$ be a finite group, $H \subseteq G$ a normal subgroup. If $H$ is strongly $R_{d}$ and $G/H$ is $R_{d}$, then $G$ is $R_{d}$.
\end{proposition}

Thus, for example, a product of $d$ cyclic groups $\rC_{mr_{1}} \times \dots \times \rC_{mr_{d}}$, where $r_{1}$, \dots,~$r_{d}$ are prime to $d!$, is an $R_{d}$ group.

Notice that subgroups and quotients of $R_{d}$ groups are not necessarily $R_{d}$: for example, $\rC_{2} \times  \rC_{2}$ is $R_{2}$, but $\rC_{2}$ is not. Furthermore, the product of two $R_{d}$ groups is not necessarily $R_{d}$. A counterexample will appear in a forthcoming paper by the first author \cite{giulio-singularities}.

By putting together Theorem~\ref{thm:moduli}, Lemma~\ref{lem:pointed} and Theorem~\ref{thm:Rd-2}, we obtain the following.

\begin{theorem}\label{thm:pointed!}
	Let $\cM$ be a category of pointed structured spaces, $(X, p, \xi) \in \cM(\overline{k})$ a tame object such that $X$ is integral of dimension $d$. If $\underaut(X, p, \xi)$ is étale of degree prime to $d!$, then $(X, p, \xi)$ is defined over its field of moduli.
\end{theorem}

\begin{proof}
	The rational point $\bfp$ of the compression $\bX_{(X, p, \xi)}$ given by Lemma~\ref{lem:pointed} is a tame quotient singularity whose fundamental group has degree prime to $d!$ by hypothesis, hence $(\bX_{(X, p, \xi)},\bfp)$ is liftable by Theorem~\ref{thm:Rd-2}. We conclude by applying Theorem~\ref{thm:moduli}.
\end{proof}

\subsection{The proofs of Theorems \ref{thm:Rd-1} and \ref{thm:Rd-2}} \hfil

Let $G \subseteq \GL_{d}(K)$, where $K$ is an algebraically closed field of characteristic prime to $|G|$, where $G$ satisfies the hypotheses of one of the theorems; and let $(S, s)$ a singularity over a field $k$ that is stably equivalent to $(\AA^{d}_{K}/G,[0])$. By extending $K$ we may assume that $k \subseteq K$, and $(S,s)_{K}$ is equivalent to $(\AA^{d}_{K}/[0])$; under these hypotheses we need to show that $(S, s)$ is liftable.

Let $P\subset G$ be the subgroup generated by pseudoreflexions. Under the hypothesis of Theorem~\ref{thm:Rd-1}, $P$ is trivial, since either $P$ or $G$ is perfect, so that the composite $P\subseteq G\to\GL_{d}(K)\xrightarrow{\det} K^{*}$ is trivial. Under the hypothesis of Theorem~\ref{thm:Rd-2}, by Lemma~\ref{lem:tqs-closed} applied to $[\AA_{K}^{d}/G]$ we have that the fundamental group of $(S,s)$ is a quotient of $G$ and hence it is abelian of order prime to $d!$. We may thus replace $G$ with the fundamental group of $(S,s)$ and assume that $P$ is trivial by Corollary~\ref{cor:tqs-rep}. Hence, in both cases we may assume that $G$ is the fundamental group of $(S,s)$.

Let $\cG$ be the fundamental gerbe of $(S,s)$; we have $\cG_{K}\simeq\cB_{K}G$.

\subsubsection*{The proof of Theorem~\ref{thm:Rd-1}}  The gerbe $\cG$ satisfies the hypotheses of Lemma~\ref{lem:moduli}, hence $\cG$ is neutral and $(S,s)$ is liftable.

\subsubsection*{The proof of Theorem~\ref{thm:Rd-2}}

First of all, we may assume that $G$ is abelian, because of the following elementary lemma, which was pointed out to us by János Kollár.

\begin{lemma}\label{lem:abelian}
Let $K$ be a field, $d$ a positive integer, and $G \subseteq \GL_{d}(K)$ a finite subgroup whose order is not divisible by $\cha K$ and prime to $d!$. Then $G$ is abelian.
\end{lemma}

\begin{proof}
We can assume that $K$ is algebraically closed. Since $\cha K$ does not divide $|G|$ we have that $K^{d}$ decomposes as a sum of irreducible representations. But the degree of any irreducible representation divides $|G|$: this is standard in characteristic~$0$, and follows from \cite[\S15.5]{serre-linear-representations} if $\cha K > 0$. So $K^{d}$ decomposes as a sum of $1$-dimensional representations, and we obtain the result.
\end{proof}

Since $G$ is abelian, then $\cG$ is associated with a cohomology class $c\in\H^{2}(k,G')$, where $G'$ is a twisted form of $G$ over $k$. Because of this, it is enough to prove that there exists a finite field extension $k'/k$ of degree prime with $|G|$ which splits $\cG$: this would imply that $[k':k]c\in\H^{2}(k,G')$ is trivial and hence that $c$ is trivial, too. 

Let us proceed by induction on $d$, starting from the case $d = 1$, which is trivial. Assume that the theorem holds in dimension $d-1$. Let $E$ be its associated variety; since $E_{K}\simeq\PP^{d-1}/\overline{G}$ where $\overline{G}$ is the image of $G$ in $\PGL_{d}(K)$, we have that $E$ has liftable singularities by the inductive hypothesis. Thanks to Proposition~\ref{prop:var-split} it is enough to find a finite field extension $k'/k$ of degree prime with $|G|$ and such that $E(k')\neq\emptyset$. 

There exists a finite separable extension $k_{1}/k$ such that $\cG_{k_{1}}$ is isomorphic to $\cB_{k_{1}}G$, and that the characters of $G$ are defined over $k_{1}$. Let $\cN$ be the normal bundle of $\cG$ in $\widehat{S}$. The pullback of $\cN$ to $\cB_{k_{1}}G$ corresponds to a $d$-dimensional representation $V$ of $G$, with an eigenspace decomposition $V = \bigoplus_{\chi \in \widehat{G}}V_{\chi}$, where $\widehat{G}$ denotes the group of characters of $G$. 

Define a functor $\Gamma\colon \aff[k]\op \arr \catset$ as follows. If $T$ is a $k$-scheme, then $\Gamma(T)$ is the set of subbundles $\cM \subseteq \cN_{T} \arr \cG_{T}$, with the property that there exist an fppf cover $\{\phi_{i}\colon T_{i} \arr T\}$ and morphisms $\psi_{i}\colon T_{i} \arr \spec k_{1}$, such that for each $i$ there exists a $\chi \in \widehat{G}$ such that $V_{\chi} \neq 0$ and $\phi_{i}^{*}\cM = \psi_{i}^{*}V_{\chi}$ in $\phi_{i}^{*}\cN = \psi_{i}^{*}V$. 

Clearly, $\Gamma$ is an fppf sheaf. The pullback $\Gamma_{k_{1}}\colon \aff[k_{1}]\op \arr \catset$ is easily checked to be represented by the disjoint union of copies of $\spec k_{1}$, one for each $\chi$ for which $V_{\chi} \neq 0$; this implies that $\Gamma$ is represented by a finite étale scheme over $k$, of degree at most $d$, because there are at most $d$ characters $\chi$ with $V_{\chi} \neq \emptyset$. So there exists a finite extension $k'/k$ of degree at most $d$, hence prime with $|G|$, such that $\Gamma(k') \neq \emptyset$.

After replacing $k$ with $k'$, we can assume that there exists a non-zero subbundle $\cM \subseteq \cN$ whose pullback to $\cG_{K} = \cB_{K}G$ is $V_{\chi} \neq 0$. Consider the projective subbundle $\PP(\cM) \subseteq \PP(\cN) = \cE$. Calling $P$ the moduli space of $\PP(\cM)$, we have $P \subseteq E$; extending the scalars to $k_{1}$ we see that $P_{k_{1}} = \PP(V_{\chi})/G = \PP(V_{\chi})$, since the action of $G$ on $\PP(V_{\chi})$ is trivial. Hence $P$ is a Brauer--Severi variety of dimension at most $d-1$, and it has index at most $d$. This means that there exists a finite extension $k'/k$ of degree at most $d$ such that $P(k') \neq \emptyset$. Then $E(k') \neq \emptyset$, and we conclude.

\subsection{Isolated $\rC_{2}$-singularities in odd dimension}\label{subsec:C2}

The proof of Theorem~\ref{thm:Rd-2} can be adapted to prove that singularities of type $\AA^{d}/G$ are $R$-singularities for many cases that are not covered in the statement of the theorem. Here we give just one example, that has an interesting application.

\begin{theorem}\label{thm:Rd-C2-isolated}
Let $n$ be a positive integer which is not divisible by $\cha K$, and consider the standard action of $\mmu_{n}$ on $\AA^{d}_{K}$ by multiplication. If $n$ and $d$ are relatively prime, then $\AA_{K}^{d}/\mmu_{n}$ is an $R$-singularity.
\end{theorem}

If $d$ is odd, $\cha K \neq 2$, $G \subseteq \GL_{d}(K)$ is a finite subgroup of order $2$, and $\AA^{d}_{K}/G$ has an isolated singularity, then $G=\mmu_{2}{\rm Id}\subset\GL_{d}(K)$. So from the theorem we get the following.

\begin{corollary}\label{cor:isolated-C2}
An isolated, odd-dimensional $\rC_{2}$-singularity is an $R$-singularity.
\end{corollary}
 
Plugging this into our main result we get the following.

\begin{corollary}
Let $\cM$ be a category of pointed structured spaces, $(X,p,\xi)$ a tame object of $\cM(\overline{k})$, such that $X$ is $d$-dimensional and integral. Assume that the automorphism group of $(X, p, \xi)$ is cyclic of order $2$, and that $p$ is an isolated fixed point for its action on $X$. Then if $d$ is odd, $(X, p,\xi)$ is defined over its field of moduli.
\end{corollary}

Thus, for example, we get that an odd dimensional abelian variety $A$ with automorphism group as small as possible, that is, cyclic of order $2$, is defined over its field of moduli, recovering in particular Shimura's result on odd-dimensional generic abelian varieties that has already been mentioned (see \cite{shimura-field-rationality-abelian}).

\begin{proof}[Proof of Theorem~\ref{thm:Rd-C2-isolated}]
Under the hypotheses of the Theorem, the associated variety $E$ is a Brauer--Severi variety of dimension $d-1$, hence it is split by a finite extension $k'$ of $k$ of degree dividing $d$. By Proposition~\ref{prop:var-split}, $k'$ splits the fundamental gerbe $\cG$, which is banded by a twisted version of $\mmu_{n}$. The result follows from the fact that $\mmu_{n}$ is abelian of degree prime to $[k':k]$.
\end{proof}

If $d$ is a positive integer, what finite groups have the property that if $G$ acts linearly on $\AA^{d}_{K}$, where $K$ is algebraically closed, with characteristic not dividing $|G|$, and $(\AA^{d}_{K})^{G} = \{0\}$, we have that $\AA^{d}_{K}/G$ has an $R$-singularity at the origin, without being $R_{d}$? We do not have any other example.

\bibliographystyle{amsalpha}
\bibliography{main}

\end{document}